\documentclass[11pt]{amsart}

\usepackage[usenames,dvipsnames,svgnames,table]{xcolor}
\usepackage[hyphens]{url}
\usepackage[pagebackref,linktocpage=true,colorlinks=true,linkcolor=Blue,citecolor=BrickRed,urlcolor=RoyalBlue]{hyperref}
\usepackage[msc-links,abbrev]{amsrefs}
\usepackage{amsmath,amsthm,amssymb}
\usepackage{mathtools}

\newtheorem{theorem}{Theorem}[section]
\newtheorem{lemma}[theorem]{Lemma}
\newtheorem{proposition}[theorem]{Proposition}

\theoremstyle{definition}
\newtheorem{note}[theorem]{Note}

\DeclareMathOperator{\tr}{tr\vphantom{g}}

\DeclareMathOperator{\sgn}{sgn}
\DeclareMathOperator{\dist}{dist}

\newcommand{\R}{\mathbf R}
\newcommand{\Sph}{\mathbf S}
\newcommand{\B}{\mathbf B}
\newcommand{\cE}{\mathcal E}
\newcommand{\bcE}{\overline{\cE}}
\newcommand{\eps}{\varepsilon}
\newcommand{\1}{\mathbf 1}
\newcommand{\C}{\mathcal{C}}

\renewcommand{\tilde}{\widetilde}

\title{The Cartan--Hadamard conjecture in dimension five}

\author{Shibing Chen}
\address{School of Mathematical Sciences, University of Science and Technology of China, Hefei, Anhui 230026, China}
\email{chenshib@ustc.edu.cn}

\author{Mohammad Ghomi}
\address{School of Mathematics, Georgia Institute of Technology,
Atlanta, GA 30332}
\email{ghomi@math.gatech.edu}
\urladdr{ghomi.math.gatech.edu}

\author{Peng Wang}
\address{School of Mathematics and Statistics, FJKLAMA, Key Laboratory of Analytical Mathematics and Applications, Fujian Normal University, Fuzhou, China}
\email{pengwang@fjnu.edu.cn,\;netwangpeng@163.com}

\date{\today\,(Last Typeset)}


\begin{document}

\begin{abstract}
We show that the sharp Euclidean isoperimetric inequality holds for domains in complete
simply connected Riemannian $5$-manifolds of nonpositive sectional
curvature, which establishes the Cartan--Hadamard conjecture in that dimension.  The main step is a sharp inequality
for constant-mean-curvature hypersurfaces, proved via
integrals over pairs of boundary points, in the spirit of Banchoff--Pohl,
together with an estimate for Jacobi fields along geodesic chords. The inequality persists for boundaries of isoperimetric regions in geodesic balls, whose mean curvature is constant only on the free part. An isoperimetric-profile argument, after Kleiner, completes the proof. Our method also gives a new proof in dimension $3$.
\end{abstract}

\maketitle

\section{Introduction}

A Cartan--Hadamard manifold $M^n$ is a complete simply connected
Riemannian $n$-manifold of nonpositive sectional curvature. The
Cartan--Hadamard conjecture \cite{aubin1976,burago-zalgaller1988,gromov1999}
asserts that domains in $M^n$ satisfy the Euclidean
isoperimetric inequality. We prove the conjecture in dimension $n=5$, the
first case not previously known. Let $\B^n$
be the unit ball in Euclidean space $\R^n$, and $\Sph^{n-1}:=\partial \B^n$ be the unit sphere. 

\begin{theorem}\label{thm:main}
Let $\Omega\subset M^5$ be a 
bounded set, and $\Gamma\coloneqq\partial\Omega$. Suppose that the perimeter $|\Gamma|$ is finite and the volume $|\Omega|>0$. Then
\begin{equation}\label{eq:main}
\frac{|\Gamma|^5}{|\Omega|^4}
\geq
\frac{|\Sph^4|^5}{|\B^5|^4},
\end{equation}
with equality only if $\Omega$ is isometric to a Euclidean ball.
\end{theorem}

The Cartan--Hadamard conjecture had been proved earlier only in dimensions $2$, $3$, and $4$
by Weil \cite{weil1926}, Kleiner \cite{kleiner1992}, and Croke
\cite{croke1984}, respectively. Related refinements and alternative
approaches include \cite{beckenbach-rado1933,schulze2008,kloeckner-kuperberg2019}.
Recent results have established the conjecture for local perturbations of the Euclidean metric
\cite{ghomi-stavroulakis2026}, for Cartan--Hadamard manifolds with
 large nullity \cite{ghomi2026-nullity}, and under  pinched negative curvature
in dimension
$5$ \cite{ghomi2026-pinched}. See \cite{ghomi-spruck2022,kloeckner-kuperberg2019,ritore2023} for more references and background.
The main step in the proof of Theorem~\ref{thm:main} is a sharp
inequality for constant-mean-curvature (CMC) hypersurfaces:

\begin{theorem}\label{thm:CMC}
Let $\Gamma^4\subset M^5$ be a smooth compact embedded hypersurface.
Suppose that the mean curvature $H$ of $\Gamma$ is a positive constant. Then
\begin{equation}\label{eq:CMC}
|\Gamma|\geq\left(\frac{4}{H}\right)^4|\Sph^4|.
\end{equation}
Equality holds only if $\Gamma$ bounds a Euclidean ball
of radius $4/H$.
\end{theorem}

The proof of Theorem~\ref{thm:CMC} involves integrals
over $\Gamma\times\Gamma$. Banchoff--Pohl \cite[Thm.~1]{BP} showed that for a
closed embedded hypersurface $\Gamma\subset\R^n$ the enclosed volume is bounded
above by a constant times $\int_{\Gamma\times\Gamma}r^{2-n}$, where
$r$ is the length of (geodesic) chords of $\Gamma$, with equality only for
spheres. Hoisington \cite[Thm.~1.1]{hoisington2021} extended this inequality to
Cartan--Hadamard manifolds, using Jacobi fields along the chords of $\Gamma$; see also \cite{teufel1993,ghomi-howard2015,hoisington-mcgrath2022}. The novel feature of our work is to use singular functions of $r$ to estimate mean curvature.

We first establish \eqref{eq:CMC} in $\R^5$ (Section~\ref{sec:euclidean}), by combining two integral formulas: a Minkowski--Green identity generated by the weight function $2\log r-2/r^2$, through which the mean curvature enters,
and a degree formula for radial projection, which encodes that
$\Gamma$ encloses a domain. These identities extend to Cartan--Hadamard manifolds via Jacobi fields along chords of $\Gamma$ (Section~\ref{sec:two-point}). Combining  these generalized identities and using an estimate for Jacobi fields (Section~\ref{sec:CMC}) yields \eqref{eq:CMC}.

To obtain \eqref{eq:main}, we then use the isoperimetric-profile
method of Kleiner \cite{kleiner1992}, as developed in higher
dimensions by Ghomi--Spruck \cite{ghomi-spruck2022}, which reduces
\eqref{eq:main} to the mean-curvature estimate \eqref{eq:CMC} for
isoperimetric regions trapped inside geodesic balls. In dimensions $n\leq 7$, such a region has
$\C^{1,1}$ boundary, which may touch the enclosing sphere, and its
mean curvature is constant only on the free part.
We will show that \eqref{eq:CMC} persists for
these regions (Section~\ref{sec:obstacle}), and \eqref{eq:main} follows.

The special role of dimension $5$ here enters through
a calibration scheme used to combine the Minkowski--Green and degree identities (Section~\ref{sec:flat-remainder}), which, with weights depending only on the length
of the chords, is confined to dimensions
$3$ and $5$ (Note~\ref{note:dimensions}). So our method also yields
a new proof of the isoperimetric inequality in dimension $3$ (Note~\ref{note:dimension-three}).
In a sequel to this work \cite{chen-ghomi-wang-general}, however,
we will show that this scheme extends to dimension $4$
and several higher dimensions by allowing the weights
to depend also on the angles between the chords
and $\Gamma$, and by generalizing the degree identity.

\section{The Euclidean Case}\label{sec:euclidean}

When $M=\R^5$, each component of  $\Gamma$ in
Theorem~\ref{thm:CMC} is a sphere of radius $4/H$ by Alexandrov's
theorem \cite{alexandrov1962}, and thus the CMC inequality \eqref{eq:CMC} is immediate. Here we give a direct proof, which extends to nonpositive curvature.

Throughout this work, unless noted otherwise, we assume that $\Omega\subset M$ is a
\emph{domain}, i.e., an open set with compact closure, and that
$\Gamma\coloneqq\partial\Omega$ is a \emph{smooth} ($\C^\infty$) embedded
hypersurface, with outward unit normal $\nu$. A \emph{chord} of $\Gamma$ is a geodesic segment between a pair of its points. The \emph{mean curvature} $H$ is the sum of the principal curvatures of $\Gamma$ with respect to $\nu$. For a submanifold $X\subset M$, we write $|X|$ for the measure of
$X$ in its own dimension.

\subsection{Overview}\label{sec:flat-overview}

We begin by considering Theorem \ref{thm:CMC} for hypersurfaces
$\Gamma^{n-1}\subset\R^n$, $n\geq3$. The restriction to $n=5$ enters
only in Sections~\ref{sec:flat-five} and \ref{sec:F}. By scaling, it suffices to consider $H=n-1$, and to prove
that then $|\Gamma|\geq|\Sph^{n-1}|$.
To motivate our approach, let us recall the classical route
to this inequality, which is based on the following equations:
$$
\int_\Gamma H\langle x,\nu\rangle=(n-1)|\Gamma|, \qquad \qquad\int_\Gamma\langle x,\nu\rangle=n|\Omega|.
$$
The first is Minkowski's classical formula \cite{hsiung1954,montiel-ros2009}, and the second follows quickly from the divergence theorem.
Since $H=n-1$, these formulas give $|\Gamma|=n|\Omega|$. The isoperimetric inequality then yields $|\Gamma|\geq|\Sph^{n-1}|$.
 This route is not available to us, since here the CMC inequality is
meant to imply the isoperimetric inequality. Instead we will develop
a pair of substitutes for the formulas above, whose combination
yields the CMC inequality in $\R^5$ directly. Both
substitutes arise from the divergence theorem, but involve chord lengths of $\Gamma$ rather than the position vector $x$;
this feature will later allow us to transplant them to
Cartan--Hadamard manifolds.

\subsection{The Minkowski--Green identity}\label{sec:flat-green}
Write $\nu_p\coloneqq\nu(p)$,  $\nu_q\coloneqq\nu(q)$ and define the following functions on pairs of distinct points $p$, $q\in\Gamma$:
$$
r(p,q)\coloneqq|p-q|,\quad
\theta(p,q)\coloneqq\frac{q-p}{r(p,q)},\quad
u(p,q)\coloneqq\langle\theta(p,q),-\nu_p\rangle,\quad
v(p,q)\coloneqq\langle\theta(p,q),\nu_q\rangle.
$$
Thus $u$ is the cosine of the angle between the directed chord $pq$
and the inward normal at $p$, and $v$ is the cosine of the angle
between $pq$ and the outward normal at $q$.
For a function $f$ defined near $\Gamma$, let $\nabla f$ and
$\nabla^2f$ denote its ambient gradient and Hessian. Furthermore, let
$\nabla_\Gamma f\coloneqq(\nabla f)^\top$ be the
tangential component of $\nabla f$ along $\Gamma$, and $\Delta_\Gamma$ be the intrinsic Laplace--Beltrami operator. Then
\begin{equation}\label{eq:laplacian-convention}
\Delta_\Gamma f
=
\tr_\Gamma\nabla^2f-H\langle\nabla f,\nu\rangle,
\end{equation}
where $\tr_\Gamma\nabla^2f$ is the trace of
$\nabla^2f$ restricted to the tangent bundle $T\Gamma$, i.e.,
$\Delta f-\nabla^2f(\nu,\nu)$.
The weight functions $\psi$ we consider for the chord inequalities below
are smooth functions on $(0,\infty)$ with a pole at $0$ such that
\begin{equation}\label{eq:psi-pole}
\psi'(r)=C_\psi\,r^{2-n}+O(r^{3-n}),
\qquad
\psi''(r)=(2-n)C_\psi\,r^{1-n}+O(r^{2-n}),
\end{equation}
for some constant $C_\psi>0$. Consequently, up to an
additive constant, $\psi(r)=C_\psi\log r+O(r)$ for $n=3$, and
$
\psi(r)=-C_\psi r^{3-n}/(n-3)+o(r^{3-n})
$
for $n\geq4$. The role of the pole here is to
produce nonzero integrals whose value is independent of $\Gamma$; Banchoff--Pohl \cite[\S5]{BP} considered the pole-free
weights $\psi(r)=r^k$, $k\geq1$. Fix $q\in\Gamma$, and set
$r_q(\cdot)\coloneqq|\cdot - q|$,
$\psi_q\coloneqq\psi\circ r_q$.
For $p\neq q$ we have $\nabla r_q(p)=-\theta$ and
$
\nabla^2r_q(p)(w,w)
=(|w|^2-\langle\theta,w\rangle^2)/r.
$
Tracing over $T_p\Gamma$, since $H=n-1$,
\eqref{eq:laplacian-convention} gives
\begin{equation}\label{eq:flat-Ap-general}
\tilde{\Delta}_\Gamma\psi_q
\coloneqq
\Delta_\Gamma\psi_q\Big|_{\Gamma\setminus\{q\}}
=\psi''(r)(1-u^2)
+\psi'(r)\Big(\frac{n-2+u^2}{r}-(n-1)u\Big).
\end{equation}
Near $q$ the pole dominates: since $u=O(r)$ on
$\Gamma$, $|\nabla_\Gamma r_q|=\sqrt{1-u^2}\to1$, and
the level set $\{r_q=\eps\}\subset\Gamma$ is close to a
sphere of radius $\eps$ in $T_q\Gamma$; hence the flux
of the leading term $C_\psi r_q^{\,2-n}\nabla_\Gamma r_q$
of $\nabla_\Gamma\psi_q$ through $\{r_q=\eps\}$, with
respect to the conormal pointing toward $q$, tends to
$-C_\psi|\Sph^{n-2}|$, while that of the $O(r^{3-n})$
term in \eqref{eq:psi-pole} tends to zero. Since $u=O(r)$
and the terms of order $r^{1-n}$ in $\psi''+(n-2)\psi'/r$
cancel by \eqref{eq:psi-pole},
$\tilde{\Delta}_\Gamma\psi_q=O(r^{2-n})$, which is integrable and defines an absolutely continuous measure
$\tilde{\Delta}_\Gamma\psi_q\,dA_\Gamma$, where $dA_\Gamma$ is the area
element of $\Gamma$. Thus,
\begin{equation}\label{eq:flat-distributional}
\Delta_\Gamma\psi_q
=\tilde{\Delta}_\Gamma\psi_q\,dA_\Gamma
+C_\psi|\Sph^{n-2}|\,\delta_q
\end{equation}
in the sense of distributions, where $\delta_q$ is the
Dirac mass at $q$. By the divergence theorem,
\begin{equation}\label{eq:flat-green-mass}
\int_\Gamma \tilde{\Delta}_\Gamma\psi_q=-C_\psi|\Sph^{n-2}|.
\end{equation}
The same computation with $\psi(r)=r^2/2$ yields Minkowski's
formula $\int_\Gamma\big((n-1)-H\langle x-q,\nu\rangle\big)=0$,
while \eqref{eq:flat-distributional} exhibits the point-mass
singularity characteristic of a Green function. Hence we call
\eqref{eq:flat-green-mass} the \emph{Minkowski--Green identity}.

\subsection{The degree identity}\label{sec:flat-degree}

Fix $p\in\Gamma$, and let $P_p(q)\coloneqq(q-p)/|q-p|$
be radial projection onto the unit sphere
$\Sph_p\subset T_p\R^n$, so that $P_p(q)=\theta(p,q)$.
Let
$$
\Sph_p^-\coloneqq\big\{\theta\in\Sph_p:\langle\theta,\nu_p\rangle<0\big\}
$$
be the open hemisphere centered at $-\nu_p$. Then
$\theta\in\Sph_p^-$ if and only if the ray $p+t\theta$, $t>0$,
initially enters $\Omega$. This ray meets $\Gamma$ in $P_p^{-1}(\theta)$,
and at a transverse intersection $q$ the sign of
$v(p,q)=\langle\theta,\nu_q\rangle$ records whether the
ray is leaving $\Omega$ ($v>0$) or entering it ($v<0$).
For almost every $\theta$ the intersections are
transverse and finite, and since exits and entries
alternate along the ray,
\begin{equation}\label{eq:crossing-number}
\sum_{q\in P_p^{-1}(\theta)}\sgn v(p,q)
=\1_{\Sph_p^-}(\theta),
\end{equation}
the indicator function of 
$\Sph_p^-$. This formula is a special case of the
expression for the winding number of $\Gamma$ about a
point, cf.
\cite[(3.20)]{hoisington2021}; here the base point lies
on $\Gamma$. Since $dP_p$ at $q$ is $r^{-1}$ times the orthogonal
projection $T_q\Gamma\to\theta^\perp$, whose determinant
is $\langle\nu_q,\theta\rangle=v$, $\textup{Jac}(P_p)=v/r^{n-1}$ on $\Gamma\setminus\{p\}$.
Hence, by the area formula applied to
$u^{n-2}$, which depends on $q$ only through $\theta$,
$$
\int_\Gamma u^{n-2}\frac{v}{r^{n-1}}
=\int_{\Sph_p}\langle\theta,-\nu_p\rangle^{n-2}
\sum_{q\in P_p^{-1}(\theta)}\sgn v(p,q)
=\int_{\Sph_p^-}\langle\theta,-\nu_p\rangle^{n-2}
=\frac{|\Sph^{n-1}|}{2^{\,n-1}},
$$
where the integral over $\Gamma$ is with respect to $q$,
with $p$ fixed; the last equality follows by evaluating
the left side on the unit sphere, on which $u=v=r/2$.
Integrating over $p\in\Gamma$ gives
$\int_{\Gamma\times\Gamma} u^{n-2}v/r^{n-1}=2^{1-n}|\Sph^{n-1}|\,|\Gamma|$. Switching $p$ and
$q$ interchanges $u$ and $v$ and leaves $r$ unchanged. Adding the two identities,
\begin{equation}\label{eq:flat-degree}
\int_{\Gamma\times\Gamma}\frac{uv(u^{n-3}+v^{n-3})}{r^{n-1}}
=2^{2-n}|\Sph^{n-1}|\,|\Gamma|.
\end{equation}
This formula, which we call \emph{the degree identity},
is the analogue of the divergence formula
$\int_\Gamma\langle q-p,\nu_q\rangle=n|\Omega|$, integrated
over $p$: for fixed $p$ both sides compute the flux
through $\Gamma$ of a radial vector field about $p$, here
$X(x)\coloneqq\langle\theta,-\nu_p\rangle^{n-2}\,\theta/r^{n-1}$
with $\theta=(x-p)/|x-p|$ and $r=|x-p|$. The factor
$r^{1-n}$ makes $X$ divergence-free, so its flux is a
universal constant rather than $n|\Omega|$.

\subsection{Ansatz}\label{sec:flat-remainder}
Since the integrals in \eqref{eq:flat-green-mass} and
\eqref{eq:flat-degree} are constant multiples of $|\Gamma|$,
for any constants $a>0$ and $b$, one quickly sees that
\begin{equation}\label{eq:flat-strategy}
\int_{\Gamma\times\Gamma}\Big(a+\tilde{\Delta}_\Gamma\psi_q+\tilde{\Delta}_\Gamma\psi_p+b\,\frac{uv(u^{n-3}+v^{n-3})}{r^{n-1}}\Big)
=a|\Gamma|\big(|\Gamma|-L\big),
\end{equation}
where
$$
L\coloneqq\frac{1}{a}\Big(2C_\psi|\Sph^{n-2}|-2^{2-n}b\,|\Sph^{n-1}|\Big),
$$
with $C_\psi$ as in \eqref{eq:psi-pole}.
If the integrand is nonnegative, then $|\Gamma|\geq L$.
Applying this to $\Gamma=\Sph^{n-1}$, which has $H=n-1$, gives
$L\leq|\Sph^{n-1}|$, with equality exactly when the integrand vanishes
on $\Sph^{n-1}$. 

So we seek $\psi$, $a$, and $b$ for which
the integrand in \eqref{eq:flat-strategy} is nonnegative and vanishes on $\Sph^{n-1}$; the
resulting bound $|\Gamma|\geq|\Sph^{n-1}|$ is then sharp. 
On $\Sph^{n-1}$, $u=v=r/2$, so by
\eqref{eq:flat-Ap-general} the two Minkowski--Green terms are
equal, and the integrand becomes
$a+2\tilde{\Delta}_{\Sph^{n-1}}\psi_q+2^{2-n}b$.
Its vanishing thus means
$\tilde{\Delta}_{\Sph^{n-1}}\psi_q
\equiv-\tfrac12\big(a+2^{2-n}b\big)$, a constant
on $\Sph^{n-1}\setminus\{q\}$.   With the pole conditions \eqref{eq:psi-pole}, this
determines $\psi$ up to an additive constant: by \eqref{eq:flat-distributional}
$$
\Delta_{\Sph^{n-1}}\psi_q
=C_\psi|\Sph^{n-2}|
\Big(\delta_q-\frac{dA}{|\Sph^{n-1}|}\Big),
$$
i.e., $\psi_q$ is a \emph{Green function} of the
sphere: its Laplacian is a point mass at the pole,
corrected by the constant needed to ensure that
$
\int_{\Sph^{n-1}}\Delta_{\Sph^{n-1}}\psi_q=0.
$
It then remains to find $a$ and $b$.

\subsection{Calibration}\label{sec:flat-five}
In dimension $n=5$, our main concern, the Green
functions $\psi_q$ of the sphere are generated, up to an additive
constant, by
$$
\psi(r)\coloneqq2\log r-\frac{2}{r^2},
$$
normalized so that $C_\psi=4$ in \eqref{eq:psi-pole}:
a direct check via \eqref{eq:flat-Ap-general} with
$u=v=r/2$ gives $\Delta_{\Sph^4}\psi_q\equiv-3$ on
$\Sph^4\setminus\{q\}$, and any two Green functions
with the same $C_\psi$ differ by a constant, since
their difference is harmonic on $\Sph^4$. Next we
determine the constants $a$ and $b$. First,
\eqref{eq:flat-Ap-general} becomes
\begin{equation}\label{eq:flat-Ap}
\tilde{\Delta}_\Gamma\psi_q=\frac{4r^2+4(r^2+4)u^2-8r(r^2+2)u}{r^4}.
\end{equation}
We set
$
\sigma\coloneqq u+v,
$
$
\tau\coloneqq u-v,
$
so that $u^2+v^2=(\sigma^2+\tau^2)/2$ and $uv=(\sigma^2-\tau^2)/4$. Multiplying by
$r^4$ and substituting \eqref{eq:flat-Ap} at both endpoints, we obtain
\begin{multline*}
r^4\Big(a+\tilde{\Delta}_\Gamma\psi_q+\tilde{\Delta}_\Gamma\psi_p+b\,\frac{uv(u^2+v^2)}{r^4}\Big)\\
=\Big(\tfrac{b}{8}\sigma^4+2(r^2+4)\sigma^2-8r(r^2+2)\sigma+ar^4+8r^2\Big)
+\Big(2(r^2+4)\tau^2-\tfrac{b}{8}\tau^4\Big).
\end{multline*}
On $\Sph^4$, $u=v=r/2$, so $\sigma=r$, $\tau=0$, and the
second bracket vanishes. Nonnegativity and vanishing of
the integrand thus require the first bracket to vanish
to second order at $\sigma=r$: its derivative there is
$(b/2-4)r^3$, forcing $b=8$, and then its value is
$(a-5)r^4$, forcing $a=5$. 

\subsection{The CMC inequality}\label{sec:F}

With the values $a=5$, $b=8$ derived above, the last displayed equation becomes
\begin{equation}\label{eq:flat-F}
5+\tilde{\Delta}_\Gamma\psi_q+\tilde{\Delta}_\Gamma\psi_p+8\frac{uv(u^2+v^2)}{r^4}=\frac{F}{r^4},
\end{equation}
where
\begin{equation}\label{eq:F-definition}
F\coloneqq(r-\sigma)^2\big(5r^2+2r\sigma+\sigma^2+8\big)+\tau^2\big(2r^2+8-\tau^2\big).
\end{equation}
The two terms of $F$ have different origins. Since $|u|=\sin\vartheta_1$ and $|v|=\sin\vartheta_2$, where $\vartheta_i$ are
the angles between the chord and the tangent planes at its endpoints,
the term $\tau^2=(u-v)^2$ measures the mismatch of the signed normal
components at the endpoints; in particular, $\tau=0$ implies
$\vartheta_1=\vartheta_2$. It plays the role of the endpoint-angle defect
$2\sin^2((\vartheta_1-\vartheta_2)/2)$ in Banchoff--Pohl \cite[(4.9)]{BP}. The term $(r-\sigma)^2$ is new, and encodes $H$.
Since $|u|,|v|\leq1$, we have $|\tau|\leq2$. Moreover,
$5r^2+2r\sigma+\sigma^2+8=4r^2+(r+\sigma)^2+8>0$ and $2r^2+8-\tau^2\geq4$. Hence
$F\geq0$, with equality exactly when $u=v=r/2$.

Near the diagonal, i.e., the set $\{(p,p)\}\subset\Gamma\times\Gamma$, we have $u,v=O(r)$, so all the functions in \eqref{eq:flat-F} are
locally integrable. We may therefore integrate
\eqref{eq:flat-F} over $\Gamma\times\Gamma$. By \eqref{eq:flat-strategy}, and using $|\Sph^3|=3|\Sph^4|/4$,
\begin{equation}\label{eq:CMC-inequality}
|\Gamma|-|\Sph^4|
=\frac{1}{5|\Gamma|}\int_{\Gamma\times\Gamma}\frac{F}{r^4}
\geq0.
\end{equation}
Thus $|\Gamma|\geq|\Sph^4|$, as desired. 

\subsection{Rigidity}
If equality holds in \eqref{eq:CMC}, then by Alexandrov's
theorem each component of $\Gamma$ is a sphere, of radius $1$
since $H=4$, and $|\Gamma|=|\Sph^4|$ leaves exactly one
component. Thus $\Gamma$ bounds a unit
ball, which completes the proof of Theorem \ref{thm:CMC} for $\R^5$.

\begin{note}\label{note:dimensions}
Since the integrand of \eqref{eq:flat-strategy} is
nonnegative and vanishes on the unit-sphere chords, its
first derivatives vanish there as well. On these
chords $\tilde{\Delta}_{\Sph^{n-1}}\psi_q$ is constant,
which by \eqref{eq:flat-Ap-general} with $u=r/2$ reads
$$
\left(1-\frac{r^2}{4}\right)\psi''
+\left(\frac{n-2}{r}-\frac{2n-3}{4}r\right)\psi'
=\text{const},
$$
while differentiating the integrand in $u$ at $u=v=r/2$
gives
$
r\big(r\psi''+(n-2)\psi'\big)=\text{const}.
$
For $n\neq3$, the second relation yields
$\psi'=\alpha r^{2-n}+\beta/r$. Substitution into the
first leaves the nonconstant terms
$
\beta(n-3)/r^2-(n-1)\alpha r^{3-n}/4.
$
Since the singularity of $\psi$ at $r=0$ requires
$\alpha\neq0$, these powers can cancel only for $n=5$.
For $n=3$ we have
$\psi(r)=2\log r$. Thus the scheme above is confined to
dimensions $3$ and $5$.
\end{note}

\section{Transition to Cartan--Hadamard Manifolds}\label{sec:two-point}

Here we generalize the Euclidean Minkowski--Green and degree identities to the nonpositive curvature setting, via Jacobi fields. The
notation and terminology of Section~\ref{sec:euclidean} naturally extend to Cartan--Hadamard manifolds $M$ since there exists a unique geodesic between every pair of points of $M$, and hence chords of $\Gamma$ are well-defined.

\subsection{Jacobi fields}
The objects entering the Euclidean identities
are the derivatives of the distance at the endpoints of the chords, and
the Jacobian of the exponential map along the chords. Here we express these through the solution of the  Jacobi equation.

Fix distinct points $p,q\in\Gamma$. Let
$\gamma:[0,r]\to M$ be the unit speed geodesic from $p$ to $q$, and note that the distance function $\dist\colon M\times M\to \R$ is smooth
off the diagonal, since $M$ is Cartan--Hadamard. Set
$$
r(p,q)\coloneqq \dist(p,q),\qquad
r_q(\cdot)\coloneqq\dist(\cdot,q),\qquad
r_p(\cdot)\coloneqq\dist(p,\cdot),
$$
so that $\psi_q=\psi\circ r_q$,  $\psi_p=\psi\circ r_p$, and
$
u\coloneqq\langle\nabla r_q(p),\nu_p\rangle,
$
$
v\coloneqq\langle\nabla r_p(q),\nu_q\rangle.
$
The first-variation formula gives
$\nabla r_q(p)=-\gamma'(0)$ and $\nabla r_p(q)=\gamma'(r)$,
so
\begin{equation}\label{eq:u-v}
u=-\langle\gamma'(0),\nu_p\rangle,
\qquad v=\langle\gamma'(r),\nu_q\rangle,
\qquad -1\leq u,v\leq1.
\end{equation}
The tangential gradients are
$
\nabla_\Gamma r_q=\nabla r_q-u\nu_p,
$
$
\nabla_\Gamma r_p=\nabla r_p-v\nu_q,
$
and hence
\begin{equation}\label{eq:tangential-gradients}
|\nabla_\Gamma r_q|^2=1-u^2,
\qquad
|\nabla_\Gamma r_p|^2=1-v^2.
\end{equation}

Choose an orthonormal frame $E_1(t),\dots,E_{n-1}(t)$
perpendicular to $\gamma'(t)$ and parallel along
$\gamma$. We use this frame to identify each normal space
$\gamma'(t)^\perp$ with $\R^{n-1}$. Let $R$ denote the Riemann curvature tensor of $M$ and set
$$
K_{ij}(t)
\coloneqq\Big\langle R\big(E_j(t),\gamma'(t)\big)\gamma'(t),\,E_i(t)\Big\rangle,
$$
an $(n-1)\times(n-1)$ matrix, which is symmetric and
nonpositive: for $X=\sum_i\xi^iE_i(t)$,
$
\xi^TK(t)\xi=\langle R(X,\gamma')\gamma',X\rangle
=\operatorname{sec}_M(X\wedge\gamma')|X|^2\leq0.
$
For symmetric matrices, we write $A\geq B$ if $A-B$ is
positive semidefinite; thus $K\leq0$.

Let $Z$ be a vector field along $\gamma$ with
$Z(t)\perp\gamma'(t)$ for all $t$, regarded as a
function $Z:[0,r]\to\R^{n-1}$ via the frame $E_i$. Then $Z$
is a \emph{Jacobi field}, i.e. $Z''+R(Z,\gamma')\gamma'=0$,
if and only if
$
Z''+KZ=0.
$
Hence the Jacobi fields normal to $\gamma$ and
vanishing at $p$ are exactly $Z(t)=A(t)Z'(0)$, where
\begin{equation}\label{eq:angular-jacobi}
A''(t)+K(t)A(t)=0,\qquad A(0)=0,\qquad A'(0)=I_{n-1}.
\end{equation}

For $\xi\in T_\theta\Sph_p=\theta^\perp$, where $\Sph_p$ now denotes the unit sphere in $T_pM$, the angular variation of $\exp_p(t\theta)$ is the Jacobi field
$
Z_\xi(0)=0,
$
$
Z_\xi'(0)=\xi.
$
In the parallel frame, $Z_\xi(t)=A(t)\xi$. Therefore, in geodesic polar coordinates about $p$, the volume element
of $M$ is
$
dV=\det A(t)\,dt\,d\theta.
$
In view of this formula we define the \emph{polar Jacobian}
$$
J(p,q)\coloneqq\det A\big(r(p,q)\big),
$$
so that $J(p,q)$ is the Jacobian of the map
$\theta\mapsto\exp_p(r\theta)$ from $\Sph_p$ onto
the geodesic sphere of radius $r=r(p,q)$ about $p$,
evaluated at $\theta=\gamma'(0)$; its Euclidean value is
$r^{n-1}$.
Since $M$ has nonpositive curvature, Rauch's comparison theorem yields
\begin{equation}\label{eq:B-lower-bound}
|A(t)\xi|\geq t|\xi|,
\qquad
J\geq r^{n-1}.
\end{equation}

Define the symmetric bilinear forms on $T_pM$ and $T_qM$ by
\begin{equation}\label{eq:hessian-remainders}
\cE_p\coloneqq\nabla^2r_q(p)-\frac{g-dr_q\otimes dr_q}{r},
\qquad
\cE_q\coloneqq\nabla^2r_p(q)-\frac{g-dr_p\otimes dr_p}{r},
\end{equation}
where $g$ is the metric of $M$. In $\R^n$, $\nabla^2r_q=\frac1r(g-dr_q\otimes dr_q)$, so
$\cE_p=\cE_q=0$.
Both terms
of $\cE_p$ vanish on $\nabla r_q(p)=-\gamma'(0)$, so
$\cE_p$ is determined by its restriction to
$\gamma'(0)^\perp$, whose matrix in the frame $E_i(0)$ we
denote by $\bcE_p$: the restriction of $\nabla^2r_q(p)$
to $\gamma'(0)^\perp$ minus $I_{n-1}/r$. Similarly
$\bcE_q$ is the matrix of $\cE_q$ on $\gamma'(r)^\perp$.
By Hessian comparison, $\cE_p$, $\cE_q\geq0$.

\subsection{Estimates near the diagonal}
The next lemma records the rates at which the chord data $u$,
$v$, $\cE_p$, $\cE_q$, $J/r^{n-1}$ approach their limits as
$q\to p$. These rates make singular integrands in this paper
locally integrable, and hold at the $\C^{1,1}$ regularity required
in Section~\ref{sec:obstacle}. Let $B_\rho(p)$ denote the geodesic ball in $M$ of radius $\rho$ centered at a point $p$; all balls in this paper are closed.

\begin{lemma}\label{lem:diagonal-estimates}
Suppose $\Gamma$ is $\C^{1,1}$.
There are $r_0,C>0$, depending only on $\Gamma$ and on the ambient geometry near it, such that for $0<r=\dist(p,q)<r_0$, with $\|\cdot\|$ the operator norm,
$$
|u|+|v|\leq Cr,
\qquad \|\cE_p\|+\|\cE_q\|\leq Cr,
\qquad |J/r^{n-1}-1|\leq Cr^2,
$$
and $|\Gamma\cap B_\rho(p)|\leq C\rho^{n-1}$ for $0<\rho<r_0$.
\end{lemma}

\begin{proof}
Since $\Gamma$ is $\C^{1,1}$ and compact, there is $r_0>0$
such that, in normal coordinates centered at any
$p\in\Gamma$ with $T_p\Gamma=\R^{n-1}\times\{0\}$ and
$\nu_p=e_n$, $\Gamma\cap B_{r_0}(p)$ is contained in the
graph of a $\C^{1,1}$ function
$h\colon\{|z|<r_0\}\subset\R^{n-1}\to\R$,
$$
q=(z,h(z)),
\qquad
h(0)=0,\qquad \nabla h(0)=0,
\qquad |h(z)|\leq C|z|^2,
\qquad |\nabla h(z)|\leq C|z|,
$$
where $C$ is independent of $p$.
For such $q$, the chord $\gamma$ from $p$ to $q$ is
the segment $\gamma(t)=tq/r$, $0\leq t\leq r$, where
$r=|q|$; so $\gamma'\equiv q/r$, and
$\langle\gamma'(0),\nu_p\rangle=h(z)/r=O(r)$. Since
$\Gamma$ is $\C^{1,1}$, $\nu$ is Lipschitz, so
$|\nu_q-\nu_p|\leq Cr$; and the metric coefficients at
$q$ are $\delta_{ij}+O(r^2)$, so
$\langle\gamma'(r),\nu_q\rangle=\langle\gamma'(0),\nu_p\rangle+O(r)$.
By \eqref{eq:u-v}, therefore,
$
|u|+|v|\leq Cr.
$
Since $\Gamma$ is compact, $\|K(t)\|\leq C$ along every
chord with $r<r_0$. By \eqref{eq:angular-jacobi},
$$
A(t)=tI_{n-1}-\int_0^t(t-s)K(s)A(s)\,ds,
$$
which gives
$
A(t)=tI_{n-1}+O(t^3),
$
$
A'(t)=I_{n-1}+O(t^2).
$
The restriction of $\nabla^2r_p$ at $\gamma(t)$ to
$\gamma'(t)^\perp$ is $A'(t)A(t)^{-1}$ in the frame
$E_i(t)$; at $t=r$ this gives
$$
\bcE_q=A'(r)A(r)^{-1}-r^{-1}I_{n-1}=O(r),
\qquad
\frac{J}{r^{n-1}}=\frac{\det A(r)}{r^{n-1}}=1+O(r^2).
$$
The same calculation along the reversed chord, from $q$ to $p$, gives the bound on $\cE_p$. Finally, $\Gamma\cap B_\rho(p)$ 
lies in the graph of $h$ over $\{|z|\leq\rho\}$, whose area is at most $C\rho^{n-1}$
since $|\nabla h|\leq C\rho$ there; hence
$|\Gamma\cap B_\rho(p)|\leq C\rho^{n-1}$.
\end{proof}

By Lemma~\ref{lem:diagonal-estimates}, the singular
quantities used below satisfy
$$
\frac{|v|}{J}=O(r^{2-n}),
\qquad
\frac{|uv|}{J}=O(r^{3-n}),
\qquad
\frac{|uv|\big(|u|^{n-3}+|v|^{n-3}\big)}{J}=O(1).
$$
Furthermore, since $|\Gamma\cap B_\rho(p)|\leq C\rho^{n-1}$,
$\int_{\Gamma\cap B_\eps(p)}r^{-\alpha}=O(\eps^{\,n-1-\alpha})$
for $0\leq\alpha<n-1$. Hence all these expressions are
integrable on $\Gamma$, and the integrals over
$\Gamma\times\Gamma$ below in which they appear are
absolutely convergent.

\subsection{The general Minkowski--Green identity}\label{sec:green}

In this subsection $\psi$ denotes any function satisfying
\eqref{eq:psi-pole}. We show that the Minkowski--Green identity
\eqref{eq:flat-green-mass} carries over to Cartan--Hadamard manifolds
verbatim.
Recall that by  \eqref{eq:flat-Ap-general} and \eqref{eq:laplacian-convention},
$$
\tilde{\Delta}_\Gamma\psi_q
=\tr_\Gamma\nabla^2\psi_q-H\langle\nabla\psi_q,\nu\rangle.
$$
We first compute $\tilde{\Delta}_\Gamma\psi_q$ explicitly. From \eqref{eq:hessian-remainders},
$
\nabla^2 r_q(p)=\frac1r(g-dr_q\otimes dr_q)+\cE_p.
$
Taking the trace over an orthonormal basis of $T_p\Gamma$ and using \eqref{eq:tangential-gradients}, we find
$$
\tr_\Gamma\nabla^2 r_q
=\frac1r\big(n-1-|\nabla_\Gamma r_q|^2\big)+\tr_\Gamma\cE_p
=\frac{n-2+u^2}{r}+\tr_\Gamma\cE_p.
$$
Since $\langle\nabla r_q,\nu_p\rangle=u$, the chain rule gives
\begin{equation}\label{eq:curved-Ap-general}
\tilde{\Delta}_\Gamma\psi_q
=\psi''(r)(1-u^2)
+\psi'(r)\Big(\frac{n-2+u^2}{r}+\tr_\Gamma\cE_p-Hu\Big),
\end{equation}
which, when $H=n-1$, differs from its Euclidean counterpart
\eqref{eq:flat-Ap-general} only by the trace term.
By Lemma~\ref{lem:diagonal-estimates}, $u=O(r)$ and
$\tr_\Gamma\cE_p=O(r)$, and $H$ is bounded; since the
terms of order $r^{1-n}$ in $\psi''+(n-2)\psi'/r$ cancel
by \eqref{eq:psi-pole}, it follows that
\begin{equation}\label{eq:Ap-bound}
|\tilde{\Delta}_\Gamma\psi_q|\leq Cr^{2-n}.
\end{equation}
Since $|\Gamma\cap B_\eps(q)|\leq C\eps^{n-1}$ by
Lemma~\ref{lem:diagonal-estimates}, \eqref{eq:Ap-bound}
implies
$\tilde{\Delta}_\Gamma\psi_q\in L^1(\Gamma)$.

\begin{proposition}[General Minkowski--Green identity]\label{lem:green}
Let $\psi$ satisfy \eqref{eq:psi-pole}. For every
fixed $q\in\Gamma$,
\begin{equation}\label{eq:green-mass}
\int_\Gamma \tilde{\Delta}_\Gamma\psi_q=-C_\psi|\Sph^{n-2}|.
\end{equation}
\end{proposition}

\begin{proof}
For small $\eps>0$, $U_\eps\coloneqq\Gamma\setminus B_\eps(q)$
has $\C^1$ boundary $\partial U_\eps$, with outward conormal
$\eta_\eps\coloneqq-\nabla_\Gamma r_q/|\nabla_\Gamma r_q|$.
By \eqref{eq:psi-pole},
$$
\partial_{\eta_\eps}\psi_q
=-\psi'(\eps)|\nabla_\Gamma r_q|
=-\big(C_\psi\,\eps^{2-n}+O(\eps^{3-n})\big)|\nabla_\Gamma r_q|.
$$
Moreover $|\nabla_\Gamma r_q|=1+O(\eps^2)$, and the graph
representation of $\Gamma$ near $q$ in the proof of
Lemma~\ref{lem:diagonal-estimates} gives
$
|\partial U_\eps|=|\Sph^{n-2}|\eps^{n-2}\big(1+O(\eps^2)\big).
$
Consequently,
\begin{equation}\label{eq:green-flux}
\int_{\partial U_\eps}\partial_{\eta_\eps}\psi_q
=-C_\psi|\Sph^{n-2}|+o(1).
\end{equation}
By the divergence theorem on $U_\eps$,
$
\int_{U_\eps}\tilde{\Delta}_\Gamma\psi_q
=\int_{\partial U_\eps}\partial_{\eta_\eps}\psi_q.
$
Letting $\eps\to0$, using
$\tilde{\Delta}_\Gamma\psi_q\in L^1(\Gamma)$, yields
\eqref{eq:green-mass}.
\end{proof}

\subsection{The general degree identity}\label{sec:degree}

Now we generalize the degree identity of
Section~\ref{sec:flat-degree}. It has the same mass as its Euclidean counterpart
\eqref{eq:flat-degree}, but $r^{n-1}$ is
replaced by the polar Jacobian $J\geq r^{n-1}$. Fix
$p\in\Gamma$, and let
$$
P_p(q)\coloneqq\frac{\exp_p^{-1}(q)}{\dist(p,q)}\in\Sph_p
$$
be radial projection from $p$, as in
Section~\ref{sec:flat-degree}, so that
$P_p(q)=\gamma'(0)$ and $u=\langle P_p(q),-\nu_p\rangle$, with
$\Sph_p^-$ defined as before. The
crossing-number formula \eqref{eq:crossing-number}
holds verbatim for the geodesic rays $\exp_p(t\theta)$
(for almost every $\theta$ the fiber $P_p^{-1}(\theta)$ is
finite and meets $\Gamma$ transversely, by the area formula
applied to $\Gamma$ and to $\{\textup{Jac}(P_p)=0\}$).
Since $dP_p$ at $q$ is $A(r)^{-1}$ composed with the
orthogonal projection $T_q\Gamma\to\gamma'(r)^\perp$,
whose determinant is $\langle\nu_q,\gamma'(r)\rangle=v$,
$\textup{Jac}(P_p)=v/J$ on $\Gamma\setminus\{p\}$;
cf.\ \cite[Cor.~2.13]{hoisington2021}. Hence, by the
area formula applied to $u^{n-2}$,
\begin{equation}\label{eq:test-degree}
\int_\Gamma u^{n-2}\frac{v}{J}
=\int_{\Sph_p}\langle\theta,-\nu_p\rangle^{n-2}
\sum_{q\in P_p^{-1}(\theta)}\sgn v(p,q)
=\int_{\Sph_p^-}\langle\theta,-\nu_p\rangle^{n-2}.
\end{equation}
At $\C^{1,1}$ regularity, we apply the area formula on
$\Gamma\setminus B_\eps(p)$, where $P_p$ is Lipschitz,
and let $\eps\to0$ using the integrability of $|v|/J$
(Lemma~\ref{lem:diagonal-estimates}).

\begin{proposition}[General degree identity]\label{prop:boundary-integral}
We have
\begin{equation}\label{eq:boundary-integral}
\int_{\Gamma\times\Gamma}\frac{uv(u^{n-3}+v^{n-3})}{J}
=2^{2-n}|\Sph^{n-1}|\,|\Gamma|.
\end{equation}
\end{proposition}

\begin{proof}
By \eqref{eq:test-degree} and Section~\ref{sec:flat-degree}, $\int_\Gamma u^{n-2}v/J=2^{1-n}|\Sph^{n-1}|$.
The right side does not depend on $p$, so integrating
over $p\in\Gamma$ gives $\int_{\Gamma\times\Gamma} u^{n-2}v/J=2^{1-n}|\Sph^{n-1}|\,|\Gamma|$.
Switching $p$ and $q$ interchanges $u$
and $v$, while 
$J(p,q)=J(q,p)$ \cite[Lem.~5]{yau1975}. So the same
holds for $uv^{n-2}$. Adding the two identities proves
\eqref{eq:boundary-integral}.
\end{proof}

\section{Proof of the CMC Inequality}\label{sec:CMC}

Here we use the generalized Minkowski--Green and degree identities to extend the Euclidean CMC inequality \eqref{eq:CMC-inequality} to
Cartan--Hadamard manifolds, and so prove Theorem~\ref{thm:CMC}. Apart
from the Jacobi-field estimate of the next subsection,  we take $n=5$ throughout. By Section~\ref{sec:two-point}, the
Minkowski--Green identity carries over verbatim, while in the degree
identity the denominator $r^4$ is replaced by the Jacobian $J\geq r^4$.
The Jacobi-field estimate (Section~\ref{subsec:jacobi}) compensates for
this loss by the excess of the distance Hessians at the endpoints of the chords. Hence the integrand $F$ in \eqref{eq:CMC-inequality} extends to a nonnegative integrand
$\tilde F$ (Section~\ref{sec:G}), and  yields
the general CMC inequality (Section~\ref{sec:CMC-int}). In the equality case the Jacobi comparison is sharp along every
chord, which forces the enclosed domain $\Omega$ to be
flat, and $\Gamma$ to be a sphere (Section~\ref{sec:rigidity}).

\subsection{The Jacobi-field estimate}\label{subsec:jacobi}
The identities of Section~\ref{sec:two-point} differ
from their Euclidean forms by two curvature corrections:
the excess of $J$ over $r^{n-1}$ in the degree identity,
and the Hessian remainders $\cE_p$, $\cE_q$ in the
Minkowski--Green identity. The following estimate, valid in
every dimension, relates these quantities.

\begin{proposition}[Jacobi-field estimate]\label{lem:jacobi-endpoint}
For every distinct pair of points $p,q\in\Gamma$,
\begin{equation}
\tr_\Gamma\cE_p+\tr_\Gamma\cE_q\geq\frac{2|uv|}{r}\Big(1-\frac{r^{n-1}}{J}\Big).
\end{equation}
\end{proposition}

\begin{proof}
Let $\gamma\colon[0,r]\to M$ be the unit-speed geodesic
from $p$ to $q$, with the parallel frame, the curvature
matrix $K(t)$, and the Jacobi matrix $A(t)$ of
Section~\ref{sec:two-point}. We control both sides of
the inequality by the matrix of second
derivatives of the distance at the two endpoints of
$\gamma$.
Let $Z$ be a Jacobi field perpendicular to $\gamma$,
regarded as a function $Z\colon[0,r]\to\R^{n-1}$ via the
parallel frame, so that
$
Z''+KZ=0.
$
Let $T$ be the solution operator of the Jacobi equation, i.e., the matrix with
$$
\begin{pmatrix}Z(r)\\ Z'(r)\end{pmatrix}
=\begin{pmatrix}T_{11}&T_{12}\\ T_{21}&T_{22}\end{pmatrix}
\begin{pmatrix}Z(0)\\ Z'(0)\end{pmatrix}.
$$
The Wronskian
$
Z_1^TZ_2'-(Z_1')^TZ_2
$
is constant for any two solutions $Z_1,Z_2$. Consequently,
\begin{equation}\label{eq:symplectic-relations}
T_{11}T_{12}^T=T_{12}T_{11}^T,
\qquad T_{12}^TT_{22}=T_{22}^TT_{12},
\qquad T_{11}^TT_{22}-T_{21}^TT_{12}=I.
\end{equation}
Here $T_{12}=A(r)$, which is invertible because there are no conjugate
points, and $T_{22}=A'(r)$.
For $x,y\in\R^{n-1}$, let $Z\coloneqq Z_{x,y}$ be the unique transverse
Jacobi field satisfying
$
Z(0)=x,
$
$
Z(r)=y.
$
Solving the transfer relation and using \eqref{eq:symplectic-relations},
we obtain
$$
\begin{pmatrix}-Z'(0)\\ Z'(r)\end{pmatrix}
=\Lambda\begin{pmatrix}x\\y\end{pmatrix},
\qquad
\Lambda\coloneqq
\begin{pmatrix}
T_{12}^{-1}T_{11}&-T_{12}^{-1}\\
-T_{12}^{-T}&T_{22}T_{12}^{-1}
\end{pmatrix}.
$$
The first two identities in \eqref{eq:symplectic-relations} show that
the diagonal blocks of $\Lambda$ are symmetric.
For the Jacobi field with endpoint values $x,y$, integration by parts
gives
\begin{equation}\label{eq:endpoint-energy}
-\langle x,Z'(0)\rangle+\langle y,Z'(r)\rangle
=\int_0^r\frac d{dt}\langle Z,Z'\rangle\,dt
=\int_0^r\big(|Z'|^2-\langle KZ,Z\rangle\big)\,dt.
\end{equation}
The second-variation formula for the distance between the two endpoints
says that the diagonal blocks of $\Lambda$ are the restrictions of
$\nabla^2r_q(p)$ and $\nabla^2r_p(q)$ to $\gamma'(0)^\perp$ and
$\gamma'(r)^\perp$. This can also be read directly from
\eqref{eq:endpoint-energy} by setting first $y=0$ and then $x=0$.
Consequently, 
\begin{equation}\label{eq:Ep-transfer}
\bcE_p=T_{12}^{-1}T_{11}-r^{-1}I,
\qquad
\bcE_q=T_{22}T_{12}^{-1}-r^{-1}I;
\end{equation}
recall that $\bcE_p$, $\bcE_q$ are the restrictions of the
remainders \eqref{eq:hessian-remainders} on $\gamma'(0)^\perp$,
$\gamma'(r)^\perp$. In particular $\bcE_q=A'(r)A(r)^{-1}-r^{-1}I$,
as in the proof of Lemma~\ref{lem:diagonal-estimates}.
For arbitrary endpoint values,
$$
\int_0^r|Z'|^2\,dt
\geq\frac1r\left|\int_0^rZ'\,dt\right|^2
=\frac{|y-x|^2}{r},
$$
and $-\langle KZ,Z\rangle\geq0$ since $K\leq0$. Therefore
$
\Lambda\geq\Lambda_0
\coloneqq\frac1r\begin{pmatrix}I&-I\\-I&I\end{pmatrix}.
$
Set
$
W\coloneqq rT_{12}^{-T}.
$
Using \eqref{eq:Ep-transfer}, the matrix inequality
$r(\Lambda-\Lambda_0)\geq0$ becomes
\begin{equation}\label{eq:block-positive}
\begin{pmatrix}
r\bcE_p&(I-W)^T\\
I-W&r\bcE_q
\end{pmatrix}\geq0.
\end{equation}

To extract the Jacobian defect from \eqref{eq:block-positive},
put $\alpha\coloneqq\tr \bcE_p$, $\beta\coloneqq\tr \bcE_q$, and $N\coloneqq I-W$.
Let $N=\sum_i s_i\,\zeta_i\xi_i^{T}$ be a singular value
decomposition of $N$, with $s_i\geq0$ and $\{\xi_i\}$,
$\{\zeta_i\}$ orthonormal bases of $\R^{n-1}$, so that
$N\xi_i=s_i\zeta_i$. By \eqref{eq:block-positive},
$
s_i^2\leq r^2\langle \bcE_p\xi_i,\xi_i\rangle
\langle \bcE_q\zeta_i,\zeta_i\rangle.
$
Hence, by the Cauchy--Schwarz inequality, the trace norm
$\|\cdot\|_{S_1}$, i.e., the sum of the singular values, satisfies
$$
\|I-W\|_{S_1}
=\sum_{i=1}^{n-1}s_i
\leq r
\sqrt{\sum_{i=1}^{n-1}\langle \bcE_p\xi_i,\xi_i\rangle}
\sqrt{\sum_{i=1}^{n-1}\langle \bcE_q\zeta_i,\zeta_i\rangle}
=r\sqrt{\alpha\beta}.
$$
By \eqref{eq:B-lower-bound}, every singular value of $A(r)=T_{12}$ is
at least $r$. Hence the singular values $s_i(W)$ lie in $(0,1]$.
Moreover,
$
\det W=r^{n-1}/\det A(r)=r^{n-1}/J.
$
For $0\leq s_i\leq1$,
we have
$
1-\prod_{i=1}^{n-1}s_i
\leq\sum_{i=1}^{n-1}(1-s_i).
$
It follows that
\begin{equation}\label{eq:det-trace-norm}
1-\frac{r^{n-1}}{J}
\leq\sum_{i=1}^{n-1}(1-s_i(W))
=\|I\|_{S_1}-\|W\|_{S_1}
\leq\|I-W\|_{S_1}
\leq r\sqrt{\alpha\beta}.
\end{equation}

It remains to pass from $\tr\bcE_p$, $\tr\bcE_q$ to the traces
over $T_p\Gamma$, $T_q\Gamma$. Write
$
\nu_p=u\nabla r_q+\eta_p,
$
$
|\eta_p|^2=1-u^2.
$
Because $\cE_p$ vanishes in the radial direction,
$
\tr_\Gamma\cE_p=\alpha-\langle\bcE_p\eta_p,\eta_p\rangle.
$
Since $\bcE_p\geq0$,
$
\langle\bcE_p\eta_p,\eta_p\rangle
\leq|\eta_p|^2\tr \bcE_p=(1-u^2)\alpha,
$
and therefore
$
\tr_\Gamma\cE_p\geq u^2\alpha.
$
The same argument gives $\tr_\Gamma\cE_q\geq v^2\beta$. Combining these equations with
\eqref{eq:det-trace-norm},  and the
arithmetic--geometric mean inequality, yields
$$
\tr_\Gamma\cE_p+\tr_\Gamma\cE_q\geq u^2\alpha+v^2\beta
\geq2|uv|\sqrt{\alpha\beta}
\geq\frac{2|uv|}{r}\Big(1-\frac{r^{n-1}}{J}\Big),
$$
which completes the proof.
\end{proof}

\subsection{The general integrand}\label{sec:G}

The curved analogue $\tilde F$ of the integrand $F$ in
\eqref{eq:flat-F} is defined, for $H=4$, by replacing $r^4$ with $J$:
\begin{equation}\label{eq:curved-pointwise}
5+\tilde{\Delta}_\Gamma\psi_q+\tilde{\Delta}_\Gamma\psi_p
+8\,\frac{uv(u^2+v^2)}{J}
\eqqcolon\frac{\tilde F}{r^4}.
\end{equation}
Note that by
\eqref{eq:curved-Ap-general} and \eqref{eq:flat-Ap-general},
$\tilde{\Delta}_\Gamma\psi_q$ here exceeds its Euclidean value
\eqref{eq:flat-Ap} by $\psi'(r)\tr_\Gamma\cE_p$, which
is $2r(r^2+2)\tr_\Gamma\cE_p/r^4$ since
$\psi'(r)=2/r+4/r^3$; similarly at $q$. Moreover
$$
8\,\frac{uv(u^2+v^2)}{J}
=8\,\frac{uv(u^2+v^2)}{r^4}
-8\,\frac{uv(u^2+v^2)}{r^4}\Big(1-\frac{r^4}{J}\Big).
$$
Hence, by \eqref{eq:flat-F}, with $\sigma=u+v$ and
$\tau=u-v$ as before,
\begin{equation}\label{eq:G-definition}
\tilde F=F+2r(r^2+2)\big(\tr_\Gamma\cE_p+\tr_\Gamma\cE_q\big)
-8uv(u^2+v^2)\Big(1-\frac{r^4}{J}\Big).
\end{equation}
The corrections to $F$ are precisely the two
curvature effects of Section~\ref{sec:two-point}: the
Hessian excesses, entering with the favorable positive sign, and
the Jacobian excess $1-r^4/J$, entering with the
unfavorable negative sign.
Proposition~\ref{lem:jacobi-endpoint} was designed to
weigh the first against the second, which yields:

\begin{lemma}\label{lem:G-positive}
For all distinct points $p,q\in\Gamma$,
$
\tilde F(p,q)\geq0.
$
Equality holds only if
$
u=v=r/2,
$
$
J=r^4,
$
and
$
\tr_\Gamma\cE_p=\tr_\Gamma\cE_q=0.
$
\end{lemma}
\begin{proof}
By \eqref{eq:F-definition}, $F\geq0$, with equality exactly when
$u=v=r/2$. If $uv\leq0$, all three terms in \eqref{eq:G-definition}
are nonnegative, and equality in the lemma cannot occur in this case:
$F=0$ would force $u=v=r/2$, hence $uv=r^2/4>0$, a contradiction.
Assume therefore $uv>0$, and define
$
Q\coloneqq2(u^2+v^2)-r^2-2.
$
Proposition~\ref{lem:jacobi-endpoint} gives
$
\tr_\Gamma\cE_p+\tr_\Gamma\cE_q\geq 2uv(1-r^4/J)/r,
$
where $1-r^4/J\in[0,1)$ by \eqref{eq:B-lower-bound}.
Therefore
$$
\tilde F\geq F+4uv(r^2+2)(1-r^4/J)-8uv(u^2+v^2)(1-r^4/J)
=F-4uv(1-r^4/J)\,Q.
$$
If $Q\leq0$, then $\tilde F\geq F\geq0$.
Suppose $Q>0$. Since $0\leq1-r^4/J\leq1$, it is enough to prove
$
F-4uvQ\geq0.
$
Because $u^2+v^2\leq2$, the assumption $Q>0$ implies $r^2<2$. A single completion of squares gives
$$
F-4uvQ=(r^2+6)\tau^2
+(3r^2+10)\left(\sigma-\frac{4r(r^2+2)}{3r^2+10}\right)^2
+\frac{r^2(-r^4+10r^2+16)}{3r^2+10}>0,
$$
where the last term is positive for $r^2<2$. Hence equality cannot occur when $Q>0$.
For equality, only $uv>0$ and $Q\leq0$ remain. Equality in the first inequality of the proof forces $F=0$, hence $r=\sigma$, $\tau=0$, and $u=v=r/2$. Then $Q=-2$, and
$
\tilde F\geq8uv(1-r^4/J).
$
Thus $J=r^4$. Substitution in \eqref{eq:G-definition} gives
$\tilde F=2r(r^2+2)\big(\tr_\Gamma\cE_p+\tr_\Gamma\cE_q\big)$, so both
traces vanish.
\end{proof}

\subsection{The general CMC inequality}\label{sec:CMC-int}

First assume $H=4$. By Propositions~\ref{lem:green}
and~\ref{prop:boundary-integral}, the general Minkowski--Green and degree
identities have the same masses as their Euclidean
counterparts. Hence  \eqref{eq:flat-strategy}
holds verbatim with $J$ in place of $r^4$ in the last term. The
integrands are locally integrable across the diagonal by
Lemma~\ref{lem:diagonal-estimates}, so we may integrate
\eqref{eq:curved-pointwise} over $\Gamma\times\Gamma$, which yields
the generalization of \eqref{eq:CMC-inequality}:
\begin{equation}\label{eq:CMC-integrated}
|\Gamma|-|\Sph^4|
=\frac{1}{5|\Gamma|}\int_{\Gamma\times\Gamma}\frac{\tilde F}{r^4}
\geq0.
\end{equation}
Thus
$|\Gamma|\geq |\Sph^4|$, proving the normalized case. Rescaling yields \eqref{eq:CMC} for arbitrary positive constant $H$. As in the Euclidean setting, \eqref{eq:CMC-integrated} is a quantitative strengthening of Theorem~\ref{thm:CMC}: the deficit controls the chord integral, which is the basis of the rigidity argument below.

\subsection{Rigidity}\label{sec:rigidity}
Here we show that if $H=4$ and $|\Gamma|=|\Sph^4|$,
then $\Omega$ is isometric to $\B^5$. Equality in
\eqref{eq:CMC-integrated} gives $\tilde F=0$ almost
everywhere on $\Gamma\times\Gamma$; since $\tilde F$ is
continuous and nonnegative off the diagonal, it vanishes
for all $p\neq q$, so by Lemma~\ref{lem:G-positive}
\begin{equation}\label{eq:equality-conditions}
u(p,q)=v(p,q)=\frac{\dist(p,q)}2,
\qquad J(p,q)=\dist(p,q)^4,
\qquad \tr_\Gamma\cE_p+\tr_\Gamma\cE_q=0
\end{equation}
for all distinct $p,q\in\Gamma$. We argue in three
parts: 

\subsubsection{Convexity}
Let $\ell$ be a complete oriented geodesic, and 
$p,q\in\ell\cap\Gamma$ be distinct, with $q>p$, i.e. $q$ is positioned after $p$ with respect to the
orientation of $\ell$. Applying \eqref{eq:u-v} and \eqref{eq:equality-conditions}
to the segment $\gamma$ of $\ell$ from $p$ to $q$ gives
$$
\langle\gamma',\nu_p\rangle=-u(p,q)=-\frac{\dist(p,q)}2<0,
\qquad
\langle\gamma',\nu_q\rangle=v(p,q)=\frac{\dist(p,q)}2>0.
$$
Thus $\ell$ crosses $\Gamma$ transversely at every intersection
point, entering $\Omega$ at the earlier point of any pair and exiting
at the later one. In particular $\ell\cap\Gamma$ cannot contain three
points $p_1<p_2<p_3$: the pair $(p_1,p_2)$ makes $p_2$ an exit, while
the pair $(p_2,p_3)$ makes $p_2$ an entrance. Hence every complete geodesic meets
$\Gamma$ in at most two points.
Now suppose $\ell$ meets $\Omega$. Since $\Omega$ is open and
bounded, $\ell^{-1}(\Omega)$ is a nonempty bounded open subset of
$\R$, and each of its components is an interval whose two endpoints
lie in $\ell\cap\Gamma$. Two components would require at least three
such points. Thus every complete geodesic meets $\Omega$ in an interval, and
$\Omega$ is geodesically convex.

\subsubsection{Flatness}
Let $p,q\in\Gamma$ be distinct, $\gamma$ be the chord from $p$
to $q$, and $A(t)$ be its Jacobi matrix. By
\eqref{eq:B-lower-bound}, $A(r)^TA(r)\geq r^2I$, so each
eigenvalue of $A(r)^TA(r)$ is at least $r^2$; their product
is $\det\big(A(r)^TA(r)\big)=J^2=r^8$ by
\eqref{eq:equality-conditions}. Hence all four equal $r^2$,
and $A(r)^TA(r)=r^2I$.
Fix $w\neq 0$ and let $Z(t)\coloneqq A(t)w$, which is nonzero
for $t>0$. There,
\begin{equation}\label{eq:norm-convexity}
\frac{d^2}{dt^2}|Z|
=\frac{|Z'|^2|Z|^2-\langle Z,Z'\rangle^2}{|Z|^3}
-\frac{\langle Z,KZ\rangle}{|Z|}\geq0,
\end{equation}
so $|Z|$ is convex on $[0,r]$; since $|Z(0)|=0$ and
$|Z(r)|=r|w|$, this gives $|Z(t)|\leq t|w|$, whereas
\eqref{eq:B-lower-bound} gives the reverse inequality.
Hence $|A(t)w|=t|w|$, so $|Z|$ is linear in $t$ and
$\frac{d^2}{dt^2}|Z|=0$; as both terms in
\eqref{eq:norm-convexity} are nonnegative, both vanish,
and in particular $KZ=0$.
Since $A(t)$ is invertible for $t>0$ and $w$ is arbitrary,
$K(t)=0$. So the sectional curvatures of all
planes tangent to $\gamma$ vanish.
Fix $x_0\in\Omega$. For any unit vector $\xi\in T_{x_0}M$,
the maximal segment of the geodesic through $(x_0,\xi)$
contained in $\Omega$ has endpoints on $\Gamma$; so it is a chord, and the sectional
curvatures of all planes tangent to it vanish. Since
$\Omega$ is star-shaped with respect to $x_0$ by geodesic
convexity, a standard Jacobi field argument \cite[p. 157]{docarmo1992} shows that
$\exp_{x_0}^{-1}\colon\Omega\to T_{x_0}M\simeq\R^5$ is an isometry
onto its image with the Euclidean metric.

\subsubsection{Roundness}
By the previous step, the closure $\overline\Omega$ is isometric to a
compact convex domain in $\R^5$ bounded by $\Gamma$, since the isometry preserves geodesics and $\Omega$ is geodesically convex. As
the isometry also preserves the mean curvature, $\Gamma$ is a
closed embedded hypersurface of $\R^5$ with $H\equiv4$, so
it is a sphere by Alexandrov's theorem
\cite{alexandrov1962}, of radius $1$. Hence
$\overline\Omega$ is a unit ball.

\section{Proof of the Isoperimetric Inequality}\label{sec:obstacle}

Here we complete the proof of Theorem~\ref{thm:main}, using the CMC inequality established in Theorem \ref{thm:CMC}. Throughout this
section $M$ is $5$-dimensional. The
isoperimetric-profile argument  \cite{kleiner1992}, \cite[Thm.~7.1]{ghomi-spruck2022} reduces
the isoperimetric inequality to a mean-curvature estimate for
isoperimetric regions in a geodesic ball $B$. The boundaries of these
regions are $\C^{1,1}$, and have constant mean curvature
away from $\partial B$, so Theorem~\ref{thm:CMC} does
not apply to them directly. For a domain $\Omega\subset B$ with
boundary $\Gamma$, we call $\Gamma\cap\partial B$ the \emph{contact
set} and $\Gamma\setminus\partial B$ the \emph{free part}. We
need the following extension.

\begin{proposition}\label{prop:obstacle}
Let $B\subset M$ be a geodesic ball, and $\Omega\subset B$ be
a domain whose boundary $\Gamma$ is a compact embedded
$\C^{1,1}$ hypersurface, smooth on the free part. Suppose
that, for some constant $H_0>0$, the mean curvature $H$ of $\Gamma$
satisfies $H=H_0$ on the free part and $H\leq H_0$
almost everywhere on the contact set. Then
$
H_0^4\,|\Gamma|\geq 4^4|\Sph^4| .
$
\end{proposition}

\begin{proof}
After rescaling the metric, we
may assume that $H_0=4$, and follow the proof of
Theorem~\ref{thm:CMC}. Lemma~\ref{lem:diagonal-estimates} holds for $\C^{1,1}$
hypersurfaces, as does Proposition~\ref{prop:boundary-integral}.
Proposition~\ref{lem:jacobi-endpoint} and Lemma~\ref{lem:G-positive}
are pointwise statements about chords and tangent hyperplanes, so they apply to $\Gamma$. Applying Proposition~\ref{lem:green}
requires the following argument.

Since $\Gamma$ is $\C^{1,1}$, its normal $\nu$ is Lipschitz,
so by Rademacher's theorem $H=\operatorname{div}_\Gamma\nu$
is defined almost everywhere and essentially bounded. Fix
$q\in\Gamma$ and let $X\coloneqq\nabla_\Gamma\psi_q$, a locally
Lipschitz tangent vector field on $\Gamma\setminus\{q\}$.
In a local $\C^{1,1}$ parametrization of $\Gamma$, with
Lipschitz metric coefficients $g_{ij}$,
$$
\Delta_\Gamma\psi_q=\operatorname{div}_\Gamma X
=\frac{1}{\sqrt{\det g}}\,\partial_i\big(\sqrt{\det g}\,X^i\big),
$$
where $\sqrt{\det g}\,X^i$ is locally Lipschitz. Since the
distributional derivative of a Lipschitz function is its
almost-everywhere derivative \cite[\S4.2.3]{evans-gariepy2015},
$\Delta_\Gamma\psi_q$ is $L^\infty_{\rm loc}$  on
$\Gamma\setminus\{q\}$. At every point
where $\nu$ is differentiable, the computation of
$\operatorname{div}_\Gamma X$ that yields
\eqref{eq:laplacian-convention} for smooth $\Gamma$ applies; hence \eqref{eq:laplacian-convention}, and with it
\eqref{eq:curved-Ap-general} and \eqref{eq:Ap-bound}, hold
almost everywhere on $\Gamma\setminus\{q\}$. Finally, the
divergence theorem
$\int_{U_\eps}\operatorname{div}_\Gamma X
=\int_{\partial U_\eps}\langle X,\eta_\eps\rangle$ holds for
Lipschitz $X$ on the $\C^{1,1}$ domains $U_\eps$ of
Proposition~\ref{lem:green} \cite[\S4.3]{evans-gariepy2015},
and the flux computation \eqref{eq:green-flux} uses only
Lemma~\ref{lem:diagonal-estimates}, which was proved at
$\C^{1,1}$ regularity. Hence Proposition~\ref{lem:green}
holds for $\Gamma$.

The only new feature here is that $H\neq4$ on the contact
set, whereas $\tilde F$ was built with $H=4$. Let $\tilde F$
be given by \eqref{eq:G-definition}, which was obtained from
\eqref{eq:curved-pointwise} by \eqref{eq:curved-Ap-general}
with $H=4$. Since the only dependence of
\eqref{eq:curved-Ap-general} on $H$ is through the term
$-\psi'(r)Hu$, and $\psi'(r)=2/r+4/r^3$, we have, almost
everywhere on $\Gamma\times\Gamma$,
$$
5+\tilde{\Delta}_\Gamma\psi_q+\tilde{\Delta}_\Gamma\psi_p
+8\,\frac{uv(u^2+v^2)}{J}
=\frac{\tilde F}{r^4}
+2\,\frac{r^2+2}{r^3}\Big(\big(4-H(p)\big)u+\big(4-H(q)\big)v\Big).
$$
Integrating over
$\Gamma\times\Gamma$, using Proposition~\ref{lem:green} and the
degree identity \eqref{eq:boundary-integral} on the left, yields
$$
|\Gamma|\big(|\Gamma|-|\Sph^4|\big)
=\frac15\int_{\Gamma\times\Gamma}\frac{\tilde F}{r^4}
+\frac25\int_{\Gamma\times\Gamma}\frac{r^2+2}{r^3}
\Big(\big(4-H(p)\big)u+\big(4-H(q)\big)v\Big),
$$
where the integrals are absolutely convergent by
Lemma~\ref{lem:diagonal-estimates}. The first term on the right is
nonnegative by Lemma~\ref{lem:G-positive}. We claim that the second is
nonnegative as well. Note that $4-H\geq0$ almost everywhere on
$\Gamma$, and $4-H=0$ on the free part. Furthermore, we always
have $v(p,q)=u(q,p)$. Thus it suffices to show
that $u(p,q)\geq0$ whenever $p\in\Gamma\cap\partial B$. Since $\Gamma$
is $\C^1$ and lies in $B$, it is tangent to $\partial B$ at $p$, so
$\nu_p$ is the outward normal of $B$ at $p$. Geodesic balls in a
Cartan--Hadamard manifold are convex. Hence the geodesic $\gamma$ from
$p$ to $q$ lies in $B$, so $\langle\gamma'(0),\nu_p\rangle\leq0$, and
$u=-\langle\gamma'(0),\nu_p\rangle\geq0$ by \eqref{eq:u-v}. Hence
$|\Gamma|\geq|\Sph^4|$, and undoing the normalization gives
$H_0^4|\Gamma|\geq4^4|\Sph^4|$.
\end{proof}

Finally we establish the main result of this work. In
Theorem~\ref{thm:main}, $\Omega$ is any bounded measurable set, its
\emph{volume} $|\Omega|$ is its $5$-dimensional Hausdorff measure,  and its \emph{perimeter} $|\Gamma|$ is the $4$-dimensional
Hausdorff measure of the reduced boundary $\partial^*\Omega$
\cite{maggi2012}, which is the area of $\Gamma$ when $\Gamma$ is a $\C^1$ hypersurface. Both sides of \eqref{eq:main} are
unchanged when $\Omega$ is modified on a set of measure zero, and the
equality statement is understood accordingly: equality holds only if
such a modification of $\Omega$ is a domain isometric to a Euclidean
ball.

\begin{proof}[Proof of Theorem~\ref{thm:main}]
Fix a geodesic ball $B\subset M$ with center $o$, and $0<V<|B|$. By
\cite[Lem.~7.2]{ghomi-spruck2022}, there exists an isoperimetric
region $\Omega^*\subset B$ with $|\Omega^*|=V$, that is, a set of least
perimeter among subsets of $B$ of volume $V$. Since
$\textup{dim}(M)\leq 7$, the same lemma gives that
$\Gamma\coloneqq\partial\Omega^*$ is $\C^{1,1}$, smooth with constant
mean curvature $H_0$ on $\Gamma\setminus\partial B$, and has mean
curvature $H\leq H_0$ almost everywhere on $\Gamma\cap\partial B$.
(In \cite{ghomi-spruck2022}, $H$ denotes the average of the principal curvatures; here it
is their sum.) Moreover $H_0>0$. Indeed, let
$f\coloneqq\dist(o,\cdot)^2/2$. By \eqref{eq:laplacian-convention}
and Hessian comparison, $\nabla^2f\geq g$, we have
$$
\Delta_\Gamma f
=\tr_\Gamma\nabla^2f-H\langle\nabla f,\nu\rangle
\geq4-H\langle\nabla f,\nu\rangle
$$
almost everywhere on $\Gamma$. At contact points $\nabla f=R\nu$,
where $R$ is the radius of $B$, so $(H_0-H)\langle\nabla f,\nu\rangle$
is nonnegative there and vanishes elsewhere. Integrating over the
closed hypersurface $\Gamma$ gives
$$
0=\int_\Gamma\Delta_\Gamma f
\geq4|\Gamma|-H_0\int_\Gamma\langle\nabla f,\nu\rangle
=4|\Gamma|-H_0\int_{\Omega^*}\Delta f ,
$$
and $\Delta f\geq5$, so $H_0\geq4|\Gamma|/\int_{\Omega^*}\Delta f>0$.
Thus Proposition~\ref{prop:obstacle} applies and
yields
$$
H_0^4\,|\partial\Omega^*|\geq 4^4|\Sph^4| .
$$

Let $\mathcal I_B(V)$ denote the isoperimetric profile of $B$, so
that $\mathcal I_B(V)=|\partial\Omega^*|$ above, and let $H_0(V)$
denote the constant mean curvature of the free part of the boundary of an
isoperimetric region of volume $V$, as above. By the proof of
\cite[Thm.~7.1]{ghomi-spruck2022} and the references given there,
$\mathcal I_B$ is continuous and increasing with
$\mathcal I_B(0^+)=0$, and $\mathcal I_B'(V)=H_0(V)$ at almost every
$V$, in our normalization of the mean curvature.
Proposition~\ref{prop:obstacle} thus gives, for almost every $V$,
$$
\big(\mathcal I_B^{5/4}\big)'
=\tfrac54\,\mathcal I_B^{1/4}\,\mathcal I_B'
=\tfrac54\big(\mathcal I_B'^{\,4}\,\mathcal I_B\big)^{1/4}
\geq\tfrac54\big(4^4|\Sph^4|\big)^{1/4}
=5|\Sph^4|^{1/4}.
$$
Since $\mathcal I_B^{5/4}$ is increasing,
$
\mathcal I_B(V)^{5/4}
\geq\int_0^V\big(\mathcal I_B^{5/4}\big)'
\geq5|\Sph^4|^{1/4}\,V,
$
and $|\Sph^4|=5|\B^5|$ turns this into
$$
\mathcal I_B(V)\geq|\Sph^4|\Big(\frac{V}{|\B^5|}\Big)^{4/5},
$$
the Euclidean isoperimetric inequality for subsets of $B$. Since
every bounded set lies in some geodesic ball,
\eqref{eq:main} follows.

Finally, suppose that equality holds in \eqref{eq:main}
for a bounded set $\Omega$. Then $\Omega$ minimizes perimeter among
bounded sets of volume $|\Omega|$; in particular it is an isoperimetric
region in any geodesic ball containing it compactly, and
\cite[Lem.~7.2]{ghomi-spruck2022} provides a representative of $\Omega$
with smooth compact embedded boundary, since $\Omega$ lies compactly
in the ball and $n=5<8$, which we again denote by $\Gamma$. Since equality holds in \eqref{eq:main},
$\Omega$ also minimizes the isoperimetric deficit
$|\partial(\cdot)|-|\Sph^4|\big(|\cdot|/|\B^5|\big)^{4/5}$
among bounded sets, and the first variation of this functional
gives $H=4|\Gamma|/(5|\Omega|)$. Combined with
$|\Gamma|^5=5^5|\B^5|\,|\Omega|^4$, this yields
$H^4|\Gamma|=4^4|\Gamma|^5/(5^4|\Omega|^4)=4^4|\Sph^4|$. By
Theorem~\ref{thm:CMC}, $\Omega$ is isometric to a Euclidean ball.
\end{proof}

\begin{note}\label{note:dimension-three}
The calibration of Section~\ref{sec:flat-five} has a simpler
analogue in dimension $3$, where the Minkowski--Green identity alone
suffices. Take $\psi(r)\coloneqq2\log r$. Then $\psi'=2/r$
satisfies \eqref{eq:psi-pole} with $n=3$ and $C_\psi=2$, so
Proposition~\ref{lem:green} gives
$\int_\Gamma\tilde{\Delta}_\Gamma\psi_q=-2|\Sph^1|=-4\pi$,
and \eqref{eq:curved-Ap-general} with $n=3$ gives
\begin{gather*}
2+\tilde{\Delta}_\Gamma\psi_q+\tilde{\Delta}_\Gamma\psi_p\\
=\frac2{r^2}\Big((r-u-v)^2+(u-v)^2
+r\big(\tr_\Gamma\cE_p+\tr_\Gamma\cE_q\big)\Big)
+\frac2r\Big(\big(2-H(p)\big)u+\big(2-H(q)\big)v\Big).
\end{gather*}
The first term on the right is nonnegative, since
$\cE_p,\cE_q\geq0$. When $H\equiv2$ the last term vanishes,
and integrating over $\Gamma\times\Gamma$ yields
$|\Gamma|(|\Gamma|-4\pi)\geq0$, the $3$-dimensional
analogue of \eqref{eq:CMC}. For an isoperimetric region in
a geodesic ball, as in Proposition~\ref{prop:obstacle}, the
last term is nonnegative as well: $2-H\geq0$ almost
everywhere on $\Gamma$, $2-H=0$ off the contact set, and
$u,v\geq0$ at contact points, as shown in the proof of that
proposition. The reduction in this section therefore
applies verbatim and yields a new proof of the
$3$-dimensional Cartan--Hadamard inequality
\cite{kleiner1992}.
\end{note}

\begin{note}\label{note:willmore}
The inequality $|\Gamma|\geq4\pi$ established in
Note~\ref{note:dimension-three}, which is the analogue of Theorem~\ref{thm:CMC} in dimension $3$,
follows also from the 
inequality
$
\int_\Gamma (H/2)^2\geq4\pi,
$
which holds in Cartan--Hadamard $3$-manifolds
\cite{schulze2020}. It is not known whether
$
\int_\Gamma(|H|/(n-1))^{n-1}\geq|\Sph^{n-1}|
$
for closed hypersurfaces of Cartan--Hadamard
$n$-manifolds when $n\geq4$.   Since $H$ is constant in
Theorem~\ref{thm:CMC}, inequality \eqref{eq:CMC} is
equivalent to $\int_\Gamma(H/4)^4\geq|\Sph^4|$, and so
settles the case of this problem for embedded CMC hypersurfaces
in $M^5$. Furthermore, \eqref{eq:CMC} coincides with the
Minkowski inequality
$
\int_\Gamma H\geq(n-1)|\Sph^{n-1}|^{\frac{1}{n-1}}|\Gamma|^{\frac{n-2}{n-1}}
$
for CMC hypersurfaces, which has been conjectured for
convex hypersurfaces of Cartan--Hadamard manifolds, and
established for $n=3$ \cite{ghomi-spruck2023c,hong2026}. 
\end{note}

\section*{Acknowledgement}
The AI tools Claude (Anthropic) and ChatGPT (OpenAI) were used in the
preparation of this manuscript. The authors have reviewed all AI-assisted
content and take full responsibility for the final manuscript.

\bibliography{references}

\end{document}